\documentclass[a4paper,11pt]{article}
\usepackage{ulem}
\usepackage[T1]{fontenc}
\usepackage[utf8]{inputenc}
\usepackage{amsmath}
\usepackage{amsthm}
\usepackage{amssymb}
\usepackage{amsfonts}
\usepackage{graphics}
\usepackage{color}
\newtheorem{theo}{Theorem}%

\newtheorem{prop}{Proposition}%
\newtheorem{defi}[prop]{Definition}%
\newtheorem{lemm}[prop]{Lemma}%
\newtheorem{exam}[prop]{Example}%

\newcommand{\Bk}{\color{black}}
\newcommand{\Rd}{\color{red}}

\def\one{\mathbf 1}
\def\ac{\mathcal A}
\def\bc{\mathcal B}
\def\eef{\mathbb E}
\def\fc{\mathcal F}
\def\gc{\mathcal G}
\def\lc{\mathcal L}
\def\mc{\mathcal M}
\def\F{\mathcal F}
\def\G{\mathcal G}
\def\H{\mathcal H}
\def\K{\mathcal K}
\def\Z{\mathbf Z}
\def\zzf{\mathbf Z}
\def\P{{\mathbb P}}
\def\R{{\mathbb R}}

\def\odc{[\![}
\def\fdc{]\!]}

\begin{document}
\title{Filtrations at the threshold of standardness}
\author{Ga\"el Ceillier, Christophe Leuridan}
\maketitle

\begin{abstract}
A.~Vershik discovered that filtrations indexed by 
the non-positive integers may have a paradoxical asymptotic behaviour 
near the time $-\infty$, called non-standardness. For example, two dyadic 
filtrations with trivial tail $\sigma$-field are not necessarily 
isomorphic. Yet, any essentially separable filtration indexed by the 
non-positive integers becomes standard when sufficiently many integers 
are skipped. 

In this paper, we focus on the non standard filtrations which become 
standard if (and only if) infinitely many integers are skipped.
We call them filtrations at the threshold of standardness, since they
are as close to standardardness as they can be although they are non-standard.

Two class of filtrations are studied, first the filtrations 
of the split-words processes, second some filtrations inspired 
by an unpublished 
example of B.~Tsirelson. They provide examples which disproves some naive 
intuitions. For example, it is possible to have a standard filtration 
extracted from a non-standard one with no intermediate (for extraction) 
filtration at the threshold of standardness. It is also possible to have 
a filtration which provides a standard filtration on the even times 
but a non-standard filtration on the odd times.   
\end{abstract}


\noindent\textbf{keywords:} filtrations, standardness, split-words processes.

\section*{Introduction}

The notion of standardness has been introduced by A.~Vershik~\cite{Vershik}
in the context of decreasing sequences of measurable partitions
indexed by the non-negative integers.
Vershik's definition and characterizations of standardness have been
translated their original ergodic theoretic formulation into 
a probabilistic language by
M.~\'Emery and W.~Schachermayer \cite{Emery-Schachermayer}.
In this framework, the objects of focus
are the filtrations indexed by non-positive integers. These are the
non-decreasing sequences $(\fc_n)_{n \le 0}$ of sub-$\sigma$-fields of
a probability space $(\Omega,\ac,\P)$.

All the sub-$\sigma$-fields of $\ac$ that we will consider
are assumed to be complete and essentially separable with
respect to $\P$.
By definition, a sub-$\sigma$-field of $(\Omega,\ac,\P)$ is
separable if it can be generated as a complete $\sigma$-field by
a sequence of events, or equivalently, by some real random variable.
One can check that a sub-$\sigma$-field $\bc \subset \ac$ is separable 
if and only if the Hilbert space $L^2(\Omega,\bc,\P)$ is separable.
\Bk

Almost all filtrations that we will consider in this study have the following
property:
for each $n$, $\fc_n$ is generated by $\fc_{n-1}$ and by some random variable
$U_n$ which is independent of $\fc_{n-1}$ and uniformly distributed
on some finite set with $r_n$ elements, for some sequence $(r_n)_{n \le 0}$
of positive integers. Such filtrations are called $(r_n)_{n \le 0}$-adic.

For such filtrations, as shown by Vershik~\cite{Vershik},
standardness turns out to be tantamount to a simpler, much more
intuitive property: an $(r_n)$-adic filtration $\fc$ is standard
if and only if $\fc$ is of product type, that is, $\fc$ is the natural
filtration of some process $V = (V_n)_{n\le 0}$ where the $V_n$ are
independent random variables; in this case, it is easy to see
that the process $V$ can be chosen with the same law as $U= (U_n)_{n\le 0}$.
So, at first reading, `standard' can be replaced with `of product
type' in this introduction.

Although intuitive, the notion of product-type filtrations is not
as simple as one could believe. For example, the assumption
that the tail $\sigma$-field $\fc_{-\infty} = \bigcap_{n \le 0} \fc_n$
is trivial, and the property $\fc_n = \fc_{n-1} \vee \sigma(U_n)$ for
every $n \le 0$ do not ensure that $(\fc_n)_{n \le 0}$ is generated
by $(U_n)_{n \le 0}$. In the standard case, $(\fc_n)_{n \le 0}$ can
be generated by some other sequence $(V_n)_{n \le 0}$ of independent
random variables which has the same law as $(U_n)_{n \le 0}$.
In the non-standard case, no sequence of independent random
variables can generate the filtration $(\fc_n)_{n \le 0}$.

The first examples of such a situation were given by
Vershik~\cite{Vershik}. By modifying and generalizing one of these
examples, M.~Smorodinsky~\cite{Smorodinsky}
and \'Emery and Schachermayer~\cite{Emery-Schachermayer}
introduced the split-words processes.

The law of a split-words process depends on an alphabet $A$,
endowed with some probability measure, and a decreasing sequence
$(\ell_n)_{n \le 0}$ of positive integers (the lengths of the words) such that
$\ell_0=1$ and the ratios $r_n = \ell_{n-1}/\ell_n$ are integers.
For the sake of simplicity, we consider here only finite alphabets 
endowed with the uniform measure. 

A split-words process is an inhomogeneous Markov process
$((X_n,U_n))_{n \le 0}$ such that for every $n \le 0$:
\begin{itemize}
\item $(X_n,U_n)$ is uniform on $A^{\ell_n} \times \odc 1,r_n \fdc$.
\item $U_n$ is independent of $\fc^{(X,U)}_{n-1}$.
\item if one splits the word $X_{n-1}$ (of length $\ell_{n-1} = r_n\ell_n$)
into $r_n$ subwords of lengths $\ell_n$, then $X_n$ is the
$U_n$-th subword of $X_{n-1}$.
\end{itemize}
Such a process is well-defined since the sequence of uniform laws on
the sets $A^{\ell_n} \times \odc 1,r_n \fdc$ is an entrance law for the transition
probabilities given above.
By construction, the natural filtration $\fc^{X,U}$ of $((X_n,U_n))_{n \le 0}$
is $(r_n)_{n \le 0}$-adic. One can check that the tail $\sigma$-field
$\fc^{X,U}_{-\infty}$ is trivial. Thus, it is natural to ask whether
$\fc^{X,U}$ is standard or not. 

Whether a split-words process with lengths $(\ell_n)_{n \le 0}$ generates
a standard filtration
or not is completely characterised~: the filtration is non-standard
if and only if
$$\sum_{n}\frac{ \ln(r_{n})}{ \ell_{n}} < +\infty \hfill \quad (\Delta).$$
Note that this condition does not depend on the alphabet $A$. 

In this statement, the `if' part and a partial converse have been
proved by Vershik~\cite{Vershik} (in a very similar framework) and by
S.~Laurent~\cite{Laurent-thesis}.
The `only if' part has been proved by D.~Heicklen~\cite{Heicklen}
(in Vershik's framework) and by G.~Ceillier~\cite{Ceillier}.
The generalization to arbitrary alphabets has been performed by
Laurent in \cite{Laurent-standardness}: the characterisations
and all the results below still hold are when the
alphabet is a Polish space endowed with some probability measure. 

Although these examples are rather simple to construct, proving
the non-standardness requires sharp tools like Vershik's
standardness criterion~\cite{Vershik,Emery-Schachermayer}. 
One can also use the I-cosiness criterion of
\'Emery and Schachermayer~\cite{Emery-Schachermayer} which may be seen 
as more intuitive by probabilists. Actually, 
Laurent proved directly that both criteria are actually
equivalent. Moreover, applying these criteria to the examples above 
leads to rather technical estimations. 

Another question concerns what happens to a filtration
when time is accelerated by extracting a subsequence. Clearly,
every subsequence of a standard filtration is still standard.
But Vershik's lacunary isomorphism theorem~\cite{Vershik} states
that from {\it any} filtration
$(\fc_n)_{n \le 0}$ such that $\fc_0$ is essentially separable
and $\fc_{-\infty}$ is trivial, 
one can extract a filtration $(\fc_{\phi(n)})_{n \le 0}$ which is standard.
This striking fact is mind-boggling for anyone who is interested
by the boundary between standardness and non-standardness. A natural
question arises:
\begin{center}
when $(\fc_n)_{n \le 0}$ is not standard, how close to identity 
the increasing map $\phi$ (from $\zzf_-$ to $\zzf_-$) provided by the
lacunary isomorphism theorem can be?
\end{center}
Of course, as standardness is an asymptotic property, the extracting map
$\phi$ has to skip an infinity of times integers (equivalently,
$\phi(n)-n \to -\infty$ as $n \to -\infty$).

In~\cite{Vershik}, Vershik provides an example of a non-standard
dyadic filtration $(\fc_n)_{n \le 0}$ such that $(\fc_{2n})_{n \le 0}$ is standard.
Gorbulsky also provides such an example in~\cite{Gorbulsky}.

Using the fact that
the family of split-words filtrations is stable by extracting 
subsequences,
Ceillier exhibits in~\cite{Ceillier} an example of a non-standard filtration
$(\fc_n)_{n \le 0}$ which is as close to standardness as it can be:
every subsequence $(\fc_{\phi(n)})_{n \le 0}$ is standard as soon as $\phi$
skips an infinity of integers.

This paper is devoted to the filtrations sharing this property. We call
them filtrations {\it at the threshold of standardness}.

\subsection*{Main results and organization of the paper}

Some definitions and classical facts used in the paper are recalled 
in an annex, at the end of the paper. In the sections~\ref{split-words} 
and~\ref{Tsirelson} which are the core of the paper, two class of 
filtrations are studied, first the filtrations 
of the split-words processes, second some filtrations inspired 
by an unpublished example of B.~Tsirelson. 

\paragraph{The case of split-words filtrations}
The first part deals with split-words filtrations.

First, we characterise the filtrations at the threshold of standardness
among the split-words filtrations.

\begin{prop}~\label{characterisation of the threshold}
A split-words filtration with lengths $(\ell_n)_{n \le 0}$ is at the
threshold of standardness
if and only if
$$\sum_{n \le 0} \frac{\ln(r_{n})}{\ell_{n}} < +\infty \hfill \quad (\Delta)$$
and
$$\inf_{n \le 0} \frac{\ln(r_{n}r_{n-1})}{\ell_{n}} >0 \hfill \quad (\star).$$
\end{prop}

Next, we characterise (among the split-words filtrations)
the filtrations that cannot be extracted from
any split-words filtration at the threshold of standardness.

\begin{prop}~\label{not extracted from a filtration at the threshold}
If
$$\sum_{n \le 0} \frac{\ln(r_{n})}{\ell_{n}} = +\infty
\hfill \quad (\neg \Delta)$$
and
$$\lim_{n \to -\infty} \frac{\ln(r_{n})}{\ell_{n}} = 0
\hfill \quad (\Box),$$
then any split-words filtration with lengths $(\ell_n)_{n \le 0}$
is standard but
cannot be extracted from a split-words filtration at the threshold of
standardness.
\end{prop}

One could think that the threshold of standardness is a kind of boundary
between standardness and non-standardness. Yet, the situation is not so
simple. Indeed, proposition~\ref{not extracted from a filtration at the
threshold} provides an example
(example~\ref{no intermediate at the threshold})
of two split-words filtrations, where
\begin{itemize}
\item the first one is non-standard,
\item the second one is standard,
\item the second one is extracted from the first one,
\item yet, no intermediate filtration (for extraction) is at the
threshold of standardness.
\end{itemize}

Furthermore, we provide an example of a non-standard split-words
filtration from which no filtration at the threshold of standardness
can be extracted (example~\ref{no extracted filtration at the threshold}).
The proof relies on theorem~\ref{various types} below.

Recall that, given any
filtration $(\F_n)_{n \le 0}$
and an infinite subset $B$ of $\Z^-$, the extracted filtration
$(\F_n)_{n \in B}$ is standard if and only if the complement
$B^c = \zzf_- \setminus B$ is large enough in a certain way.
Here, the meaning of
``large enough'' depends on the filtration $\F$ considered. When
$\F$ is at the threshold of standardness, ``large enough'' means
exactly ``infinite''. But various types of transition from
non-standardness to standardness are possible, and the next
theorem provides some other possible conditions.

\begin{theo}~\label{various types}
Let $(\alpha_n)_{n \le 0}$ be any sequence of non-negative real numbers.
There exists a split-words filtration $(\fc_n)_{n \le 0}$ such that 
for every infinite subset $B$ of $\zzf_-$, the extracted filtration
$(\fc_n)_{n \in B}$ is standard if and only if
$$\sum_{n \in B^c} \alpha_n = +\infty \text{ or } \sum_{n \le 0}
\one_{[n \notin B,\,n+1\notin B]} = +\infty.$$
\end{theo}

Theorem~\ref{various types} immediately provides other interesting
examples. For example, it may happen that $(\fc_{2n})_{n \le 0}$ is standard 
while $(\fc_{2n-1})_{n \le 0}$ is not, or vice versa. When this phenomenon 
occurs, we will
say that the filtration $(\fc_{n})_{n \le 0}$ ``interlinks''
standardness and non-standardness.

Repeated interlinking is possible. By suitably slowing time suitably
in a filtration at the threshold of standardness
(example~\ref{repeated interlinking}),
one gets can a filtration $(\fc_{n})_{n \le 0}$ such that
$(\fc_{2n})_{n \le 0}$, $(\fc_{4n})_{n \le 0}$, $(\fc_{8n})_{n \le 0}$,...
are non-standard, whereas $(\fc_{2n+1})_{n \le 0}$, $(\fc_{4n+2})_{n \le 0}$,
$(\fc_{8n+4})_{n \le 0}$,... are standard.

\paragraph{Improving on an example of Tsirelson}
In a second part, we study another type of filtrations inspired by
a construction of Tsirelson in unpublished notes~\cite{Tsirelson}.

Tsirelson
has constructed an inhomogeneous discrete Markov process $(Z_n)_{n \le 0}$
such that the random variables $(Z_{2n})_{n \le 0}$ are independent
and such that the natural filtration $(\fc^Z_n)_{n \le 0}$ is non-standard 
although its tail $\sigma$-field is trivial. This example is illuminating 
since ``simple'' reasons explain why the standardness criteria do not hold 
and no technical estimates are required.
Tsirelson's construction relies on a particular structure of the
triples $(Z_{2n-2},Z_{2n-1},Z_{2n})$ that we explain. We call ``bricks'' 
these triples.

In this paper, we give a modified and simpler construction which provides
stronger results by requiring more on the bricks: in our construction,
for every $n \le 0$,
$Z_{2n-2}$ is a deterministic function of $Z_{2n-1}$
and $Z_{2n-1}$ is a deterministic function of $(Z_{2n-2},Z_{2n})$,
hence the filtration $(\fc^Z_{2n})_{n \le 0}$ is generated by the
sequence $(Z_{2n})_{n \le 0}$ of independent random variables.
Yet, $(\fc^Z_{2n-1})_{n \le 0}$ is not standard. Thus the filtration
$\fc^Z$ ``interlinks'' standardness and non-standardness. Actually,
we have a complete characterisation of the standard filtrations
among the filtrations extracted from $\fc^Z$.

\begin{theo}~\label{existence theorem}
There exists a Markov process $(Z_n)_{n \le 0}$ such that
\begin{itemize}
\item for each for $n \le 0$, $Z_n$ takes its values in some finite set $F_n$.
\item the random variables $(Z_{2n})_{n \le 0}$ are independent.
\item for each for $n \le 0$, $Z_{2n-1}$ is a
measurable deterministic function of $(Z_{2n-2},Z_{2n})$.
\item the filtration $(Z_n)_{n \le 0}$ is $(r_n)_{n \le 0}$-adic for some sequence
$(r_n)_{n \le 0}$.
\item for any infinite subset $D$ of $\zzf_-$, the filtration
$(\fc^Z_{n})_{n \in D}$ is standard if and only if $2n-1 \notin D$
for infinitely many $n \le 0$.
\end{itemize}
In particular, the filtration
$(\fc^Z_{2n-1})_{n \le 0}$ is at the threshold of standardness.
\end{theo}

In this theorem, the statement that $(\fc^Z_{2n-1})_{n \le 0}$ is
at the threshold of standardness cannot be deduced from
the standardness of $(\fc^Z_{2n})_{n \le 0}$ and the non-standardness
of $(\fc^Z_{2n-1})_{n \le 0}$ only. Indeed, the example of repeated
interlinking mentioned above (see example~\ref{repeated interlinking} 
in section~\ref{split-words})
provides a counterexample (modulo a time-translation). The proof that
$(\fc^Z_{2n-1})_{n \le 0}$ is at the threshold of standardness
actually uses the fact that $(Z_n)_{n \le 0}$ is an inhomogeneous
Markov process.
%

\section{The case of split-words filtrations}~\label{split-words}

In the whole section, excepted in subsection~\ref{interlinking}, 
$\fc = (\fc_n)_{n \le 0}$ denotes a split-words filtration
associated to a finite alphabet $A$
(endowed with the uniform measure) and a decreasing sequence
$(\ell_n)_{n \le 0}$ of positive integers (the lengths) such that
$\ell_0=1$ and the ratios $r_n = \ell_{n-1}/\ell_n$ are integers.

First, we prove the characterisation at the threshold of standardness
among the split-words filtrations stated in proposition
~\ref{characterisation of the threshold}.

\subsection{Proof of proposition~\ref{characterisation of the threshold}}

{\bf Preliminary observations:}
let $B$ be an infinite subset of $\Z^-$ such that $B^c$ is infinite.
Then the filtration $(\F_n)_{n \in B}$ is a split-words filtration with
lengths $(\ell_n)_{n \in B}$. The ratios between successive lengths are
the integers $(R_n)_{n\in B}$ given by
$$R_n = \ell_{m(n)}/\ell_n \text{ where } m(n) = \sup\{k<n~:~ k \in B\}.$$
Set
$B_1=B\cap (1+B)$ and $B_2=B\backslash(1+B).$
Then $B_2$ is infinite and
\begin{itemize}
\item for $n \in B_1$, $R_n=r_n$,
\item for $n \in B_2$, $R_n\ge r_n r_{n-1}$.
\end{itemize}

Furthermore, if $B^c$ does not contain two consecutive integers,
then for any $n \in B_2$, one has $n-2\in B$ since $n-1\notin B$,
thus $m(n)=n-2$ and $R_n= r_n r_{n-1}$.

{\bf Proof of the "if" part:} assume that
$$\sum_{n\le 0} \frac{\log(r_n)}{\ell_n}<+\infty
\text{ and } \inf_{n\le 0} \frac{\log(r_n r_{n-1})}{\ell_n} >0.$$
The first condition ($\Delta$) ensures that $\F$ is not standard.
Let $B$ be an infinite subset of $\Z^-$ such that $B^c$ is infinite. One has
$$\sum_{n\in B} \frac{\log(R_n)}{\ell_n}
\ge \sum_{n\in B_2} \frac{\log(R_n)}{\ell_n}
\ge \sum_{n\in B_2} \frac{\log(r_n r_{n-1})}{\ell_n} = +\infty,$$
since $B_2$ is infinite and
$\inf\{ (\log(r_n r_{n-1}))/\ell_n~;~n\le 0\} >0.$
Thus, the split-words filtration $(\fc_n)_{n \in B}$ is standard
since the sequence of lengths $(\ell_n)_{n \in B}$ fulfils condition
$\neg (\Delta)$. Therefore $\F$ is at the threshold of standardness.

{\bf Proof of the "only if" part:}
condition $(\Delta)$, which is equivalent to the non-standardness of $\F$,
is necessary for $\F$ to be at the threshold of standardness. Let us show
that if $(\Delta)$ and $\neg(\star)$ hold, \Bk then $\F$ is not at the
threshold of standardness. Since the reals $\log(r_n r_{n-1})/\ell_n$
are positive, condition $\neg (\star)$ induces the existence of a subsequence
$(\log(r_{\phi(n)} r_{\phi(n)-1})/\ell_{\phi(n)})_{n \le 0}$ such that
$$ \forall n \in \Z^-,\quad  \frac{\log(r_{\phi(n)} r_{\phi(n)-1})}{\ell_{\phi(n)}}
\le 2^n \quad \text{ and } \quad \phi(n-1)\le \phi(n)-2.$$

Set $B=(\phi(\Z^-)-1)^c$. Let us show that the filtration
$(\F_n)_{n\in B}$ is not standard. By construction, $\phi(\Z^-)$ is infinite
and does not contain two consecutive integers. Hence $B$ and $(B)^c$
are both infinite and $B_2=B\backslash (B+1)=\phi(\Z^-)$.
Moreover, according to the preliminary observations,
$R_n=r_n$ for every $n \in B_1$
and $R_n=r_n r_{n-1}$ for every $n \in B_2$ since $B^c$ does not
contain two consecutive integers. Thus
\begin{eqnarray*}
\sum_{n\in B} \frac{\log(R_n)}{\ell_n} &=& \sum_{n\in B_1} \frac{\log(r_n)}{\ell_n} +  \sum_{n\in \phi(\Z^-)} \frac{\log(r_n r_{n-1})}{\ell_n}\\
&\le& \sum_{n \le 0} \frac{\log(r_n)}{\ell_n} +  \sum_{m\le 0} 2^m\\
&<& +\infty.
\end{eqnarray*}
Therefore $(\F_n)_{n \in B}$ is not standard. 
Thus $\F$ is not at the threshold of standardness.

\subsection{Proof of proposition~\ref{not extracted from a filtration at the threshold} and example}

\begin{proof}
Assume that $(\neg\Delta)$ and $(\Box)$ hold and that
$\F$ is extracted from some split-words filtration $\H$
with lengths $(\ell'_n)_{n \le 0}$, namely $\F_n = \H_{\phi(n)}$ for every
$n \le 0$, for some increasing map $\phi$ from $\Z^-$ to $\Z^-$.
Then for every $n \le 0$, $\ell_n=\ell'_{\phi(n)}$ and
$r_n = r'_{\phi(n)} \cdots r'_{\phi(n-1)+1}$ where $r'_k = \ell'_{k-1}/\ell'_{k}$.
Let us show that $\H$ cannot be at the threshold of standardness.

Condition $(\neg\Delta)$ ensures that $\F$ is standard. 
If $\phi$ skips only finitely many integers, then $\H$ is standard
and the conclusion holds. Otherwise, $\phi(u_n-1)\le \phi(u_n)-2$
for infinitely many $n$, and for those $n$, 
$$\frac{\log(r'_{\phi(n)}r'_{\phi(n)-1})}{\ell'_{\phi(n)}}
\le \frac{\log(r'_{\phi(n)} \cdots r'_{\phi(n-1)+1})}{\ell'_{\phi(n)}}
= \frac{\log r_n}{\ell_n}.$$
Thus, $(\Box)$ implies that
$$\inf_{k\le 0} \frac{\log(r'_k r'_{k-1})}{\ell'_k}=0.$$
Since the sequence $(r'_n)_{n \le 0}$ does not fulfill condition $(\star)$,
$\H$ is not at the threshold of standardness.
\end{proof}

\begin{exam}~\label{no intermediate at the threshold}
Define the sequence of lengths $(\ell_n)_{n \le 0}$ by
$\ell_0=1$, $\ell_{-1}=2$ and, for every  $n\le-1$,
$$\ell_{n-1}=\ell_n 2^{\lfloor \ell_n/|n| \rfloor},$$
where $\lfloor x\rfloor$ denotes the integer part of $x$.

A recursion shows that for every $n \leq 0$, $\ell_n$ is a power of $2$,
and that $\ell_n \ge 2^{|n|} \ge |n|$, hence $r_n=\ell_{n-1}/\ell_n\ge 2$.
Moreover, for every $n \le -1$,
$$\frac{\log_2(r_n)}{\ell_n} =
\frac{{\lfloor \ell_n/|n| \rfloor}}{\ell_n} \in
\left[\frac{1}{2|n|},\frac{1}{|n|}\right].$$
Therefore, $(\neg \Delta)$ and $(\Box)$ hold, hence $\F$ is standard
but cannot be extracted from any split-words filtration at the threshold
of standardness.

Yet, since each $\ell_n$ is a power of $2$, $\F$ is extracted
from the dyadic split-words filtration $\H$, which is not standard.
Since every filtration extracted from $\H$ is a split-words filtration,
one can deduce that there is no intermediate filtration (for extraction)
between $\H$ and $\F$.
\end{exam}

Remark: there are trivial examples of standard split-words filtrations
which cannot be
extracted from any split-words filtration at the threshold of standardness.
For example, consider any split-words filtrations such that
$\neg(\Delta)$ holds and such that $r_n$ is a prime number for every
$n \le 0$. The last condition
prevents the filtration from being extracted from {\it any other}
split-words filtration. Yet, it still could be 
extracted from some filtration at the threshold of standardness
which is not a split-words filtration.

\subsection{Proof of theorem~\ref{various types}}

Replacing $\alpha_{n}$ by $\min(\max(\alpha_{n},1/|n+2|^2),1)$
for $n\le-3$ does not change the nature of the series
$\sum_{k \in B^c} \alpha_{k}$, hence we may assume that for $n\le-3$,
$$1/|n+2|^2 \le \alpha_{n} \le 1.$$
Set $\ell_0=1$, $\ell_{-1}=2$, $\ell_{-2}=8$, $\ell_{-3}=64$,
$\ell_{-4}=2^{11}=2048$ and $\ell_{n-2}=2^{\lfloor\alpha_{n-1}\ell_n\rfloor}$
for every $n\le -3$, where $\lfloor x\rfloor$ denotes the integer part of $x$.
We begin with two technical lemmas.

\begin{lemm}\label{minorations}
For every $n \le -1$, $\ell_n \ge |n|^3$ and $\ell_{n} \ge 2|n+1|^2 \ell_{n+1}$.
\end{lemm}

\begin{proof}[Proof of lemma~\ref{minorations}]
The proof of lemma~\ref{minorations} is done by induction. One checks that
the above inequalities hold for $-4 \le n \le -1$.

Fix some $n \le -3$. Assume that the inequalities hold for $n+1$, $n$ and
$n-1$. Then
\begin{eqnarray*}
\log_2\ell_{n-2} - \log_2\ell_{n-1}
&=& \lfloor\alpha_{n-1}\ell_{n}\rfloor - \lfloor\alpha_n\ell_{n+1}\rfloor\\
&\ge& \alpha_{n-1}\ell_{n}-1-\alpha_n\ell_{n+1}\\
&\ge& \frac{\ell_{n}}{|n+1|^2}-\ell_{n+1}-1\\
&\ge& \ell_{n+1}-1 \text{ (since } \ell_{n} \ge 2|n+1|^2 \ell_{n+1})\\
&\ge& |n+1|^3-1 \text{ (since } \ell_{n+1} \ge |n+1|^3),
\end{eqnarray*}
hence
$$\ell_{n-2}/\ell_{n-1} \ge 2^{|n+1|^3-1} \ge 2|n-1|^2 \text{ (since }n \le -3).$$
Since $\ell_{n-1} \ge |n-1|^3$, one has
$$\ell_{n-2} \ge 2|n-1|^2\ell_{n-1} \ge 2|n-1|^5 \ge |n-2|^3
\text{ (since }n \le -3).$$
Thus the inequalities hold for $n-2$. The proof is complete.
\end{proof}

\begin{lemm}\label{estimates}
For every $n \le -4$,
\begin{eqnarray*}
\frac{\log_2 \ell_{n-1}}{\ell_n} &\le& \frac{1}{2|n+1|^2},\\
\frac{\alpha_{n-1}}{2} \le \frac{\log_2 \ell_{n-2}}{\ell_n}
&\le& \alpha_{n-1},\\
\frac{\log_2 \ell_{n-3}}{\ell_n} &\ge& 1.
\end{eqnarray*}
\end{lemm}

\begin{proof}[Proof of lemma~\ref{estimates}]
Fix $n \le -4$. The assumptions made on the sequence $(\alpha_k)_{k \le 0}$,
and lemma~\ref{minorations} entail
$\ell_n\alpha_{n-1}\ge |n|^3/|n+1|^2\ge 1,$ thus
$\alpha_{n-1}\ell_n/2 \le \lfloor\alpha_{n-1}\ell_{n}\rfloor
\le \alpha_{n-1}\ell_n$. Thus, the recursion formula
$\ell_{n-2} = 2^{\lfloor\alpha_{n-1}\ell_n\rfloor}$ yields
$$\frac{\alpha_{n-1}}{2}\le \frac{\log_2 \ell_{n-2}}{\ell_n} \le \alpha_{n-1}.$$
Since $n \le -4$, the same inequalities hold for $n+1$ and $n-1$, hence
by lemma~\ref{minorations}
$$\frac{\log_2 \ell_{n-1}}{\ell_n}
\le \alpha_n \frac{\ell_{n+1}}{\ell_{n}} \le
\frac{\ell_{n+1}}{\ell_{n}} \le \frac{1}{2|n+1|^2},$$
and
$$\frac{\log_2 \ell_{n-3}}{\ell_n} \ge \frac{\alpha_{n-2}}{2}
\frac{\ell_{n-1}}{\ell_{n}} \ge \frac{1}{2|n|^2}2|n|^2 = 1.$$
The proof is complete.
\end{proof}

We now prove theorem~\ref{various types}.

Let us check that the split-words filtration associated to the
to the lengths $(\ell_n)_{n \le 0}$ fulfills the properties of the 
previous proposition.

Let $B$ be an infinite subset of $\Z^-$ such that $B^c$ is infinite.
Since replacing $B$ by $B \setminus \{-2,-1,0\}$ does not change the nature
of the filtration $(\F_n)_{n\in B}$, one may assume that
$B \subset ]-\infty,-3]$.

Set $m(n) = \sup\{k<n~:~ k \in B\}$ for every $n \le 0$.
Then $(\ell_{m(n)}/\ell_n)_{n \in B}$ is the sequence of ratios
associated to the lengths $(\ell_n)_{n \in B}$.
Since $(\Delta)$ characterises standardness of split-words filtrations,
$$(\F_n)_{n \in B} \text{ is standard } \Longleftrightarrow
\sum_{n \in B} \frac{\log_2 (\ell_{m(n)}/\ell_n)}{\ell_n} = +\infty
\Longleftrightarrow\sum_{n \in B} \frac{\log_2 \ell_{m(n)}}{\ell_n} = +\infty,$$
where the last equivalence follows from the convergence of
the series $\sum_n \log_2\ell_n/\ell_n$ since
$\ell_n \ge 2^{|n|}$ for every $n \le 0$.

Let us split $B$ into three subsets:
\begin{itemize}
\item $B_1=\{  n\in B : m(n)=n-1\}$,
\item $B_2=\{  n\in B : m(n)=n-2\}$,
\item $B_3=\{  n\in B : m(n)\le n-3\}$.
\end{itemize}
Then
$$\sum_{n \in B} \frac{\log_2 \ell_{m(n)}}{\ell_n}
= \sum_{n \in B_1} \frac{\log_2 \ell_{n-1}}{\ell_n}
+ \sum_{n \in B_2} \frac{\log_2 \ell_{n-2}}{\ell_n} +
\sum_{n \in B_3} \frac{\log_2 \ell_{m(n)}}{\ell_n}.$$
The inequality $\ell_{m(n)} \ge \ell_{n-3}$ for $n \in B_3$ and
lemma~\ref{estimates} show that in the right-hand side,
\begin{itemize}
\item the first sum (over $B_1$) is always finite;
\item the middle sum (over $B_2$) has the same nature as
$\sum_{n \in B_2} \alpha_n$;
\item the last sum (over $B_3$) is finite if and only if $B_3$ is finite.
\end{itemize}

When $B_3$ is finite, any pair of consecutive integers excepted a finite
number of them contain at least one element of $B$. Hence, $(B_2 -1)$ only
differs from $B^c$ by a finite set of integers. Thus the sum
$\sum_{n \in B_2} \alpha_n$ has the same nature as $\sum_{n \in B^c} \alpha_n$.
Theorem~\ref{various types} follows.

\subsection{Some applications of theorem~\ref{various types}}

Choosing particular sequences $(\alpha_n)_{n \le 0}$ in
theorem~\ref{various types}
provides interesting examples of non-standard filtrations.
In what follows, $\F$ denotes
the filtration associated the sequence
$(\alpha_n)_{n \le 0}$ given by theorem~\ref{various types}.

\begin{exam}
If $\alpha_n=1$ for every $n$, then $\F$ is at the threshold
of standardness.
\end{exam}

\begin{exam}
If $\alpha_{n}=0$ for every even $n$ and $\alpha_{n}=1$ for every odd $n$,
then $(\F_{2n})_{n \le 0}$ is standard whereas $(\F_{2n-1})_{n \le 0}$ is not.
\end{exam}

\begin{exam}
If the series $\sum \alpha_n$ converges, then for every infinite subset
$B$ of $\Z^-$, the extracted filtration $(\F_n)_{n\in B}$ is standard
if and only if $(B \cup (B-1))^c$ is infinite. In particular, the
filtrations $(\F_{2n})_{n \le 0}$ and  $(\F_{2n-1})_{n \le 0}$ are at the
threshold of standardness.
\end{exam}

\begin{exam}~\label{no extracted filtration at the threshold}
If $\alpha_n\sim 1/|n|$ as $n$ goes to $-\infty$, then $\F$ is not standard
and no filtration at the threshold of standardness can be
extracted from $\F$.
\end{exam}

\begin{proof}[Proof of example~\ref{no extracted filtration at the threshold}]
The non-standardness of $\F$ is immediate by theorem~\ref{various types}.

Call $\mu$ the non-finite positive measure on $\Z^-$ defined by
$$\mu(B) = \sum_{n \in B} \alpha_n \text{ for } B \subset \Z^-.$$

Let $(\F_n)_{n \in B}$ be any non-standard filtration extracted from $\F$.
We show that $(\F_n)_{n \in B}$ cannot be at the threshold of standardness
by constructing a subset $B'$ of $B$ such that
and $(\F_n)_{n \in B'}$ is not standard although $B \setminus B'$ is infinite .

By to theorem~\ref{various types}, we know that $\mu(B^c) < +\infty$ and
$$n \notin B \text{ and } n+1 \notin B
\text{ only for finitely many } n \in \Z^-.$$
Since $\mu(B^c)$ is finite, the elements of $B^c$ get rarer and rarer as
$n \to -\infty$.
In particular, the set $A = (B-1) \cap B \cap (B+1)$ is infinite.

We get $B'$ from $B$ by removing a ``small'' infinite subset of $A$.
Namely, we set $B'= B \backslash A'$ where $A'$ is an infinite
subset of $A$ which does
not contain two consecutive integers and chosen such that $\mu(A') < +\infty$.
By construction, $B \backslash B' = A'$ is infinite and
$\mu((B')^c) < +\infty$ since $(B')^c = B^c \cup A'$. Thus $B'$ is
an infinite subset of $B$.

Using the definition of $A$ and the fact that $A'$ does not contain
two consecutive integers and by construction of $A$, one checks that
$(B'\cup(B'-1))=(B\cup(B-1))$, therefore $(B'\cup(B'-1))^c$ is infinite.

Thus $(\F_n)_{n \in B'}$ is not standard, which shows that
$(\F_n)_{n \in B}$ is not at the threshold of standardness.
\end{proof}


\subsection{Interlinking standardness and non standardness}
~\label{interlinking}


Given any filtration $(\fc_n)_{n \le 0}$, a simple way to get a ``slowed''
filtration is to repeat each $\fc_n$ some finite number of times, which
may depend of $n$. We now show that this procedure does not change the
nature of the filtration.

\begin{lemm}\label{slowed filtrations}
Let $(\F_n)_{n \le 0}$ be any filtration and $\phi$ an increasing map
from $\Z^-$ to $\Z^-$ such that $\phi(0)=0$. For every $n \le 0$,
set $\G_n=\F_k$ if $\phi(k-1)+1 \le n \le \phi(k)$. Then:
\begin{itemize}
\item $(\G_n)_{n \le 0}$ is a filtration,
\item $(\F_n)_{n \le 0}$ is extracted from $(\G_n)_{n \le 0}$,
\item $(\G_n)_{n \le 0}$ is standard if and only if $(\F_n)_{n \le 0}$ is standard.
\end{itemize}
\end{lemm}

\begin{proof}[Proof of lemma~\ref{slowed filtrations}]
By construction, $\G_{\phi(k)} = \F_k$ for every $k \le 0$ and the
sequence $(\G_n)_{n \le 0}$ is constant on every interval
$\odc \phi(k-1)+1,\phi(k) \fdc$. The first two points follow.

The ``only if'' part of the third point is immediate since $\F$ is extracted
from $\G$.

Assume that $\F$ is standard. Then, up to a enlargement of the probability
space, one may assume that $\F$ is immersed in some product-type filtration
$\H$. Define a slowed filtration by $\K_n=\H_k$ if
$\phi(k-1)+1 \le n \le \phi(k)$. Then $\K$ is still a product-type filtration
To prove that $\G$ is immersed in $\K$, we have to check that for every
$n \le -1$, $\G_{n+1}$ and $\K_n$ are independent conditionally on $\G_n$.
This holds in any case since:
\begin{itemize}
\item when $\phi(k-1)+1 \le n \le \phi(k)-1$, $\G_{n+1}=\F_k$, $\K_n=H_k$
and $\G_n=\F_k$;
\item when $n=\phi(k)$, $\G_{n+1}=\F_{k+1}$, $\K_n=H_k$ and $\G_n=\F_k$.
\end{itemize}
Hence $\G$ is standard.
\end{proof}

\begin{exam}~\label{repeated interlinking}
Assume that $(\F_n)_{n \le 0}$ is at the threshold of standardness. Set
$\phi(0)=0$, $\phi(-1)=-1$ and, for every $k \le 0$,
$\phi(2k) = -2^{|k|}$ and $\phi(2k-1) = -2^{|k|}-1$.
Let $\G$ be the slowed filtration obtained from $\F$ as above.
Then for any $d \ge 1$, the filtration $(\G_{2^dn})_{n \le 0}$ is not standard,
whereas the filtration $(\G_{2^dn-2^{d-1}})_{n \le 0}$ is standard.
\end{exam}

\begin{proof}[Proof of example~\ref{repeated interlinking}]
Fix $d \ge 1$. The filtrations $(\G_{2^dn})_{n \le -2}$ and
$(\G_{2^dn-2^{d-1}})_{n \le -1}$ can be obtained from $(\F_{n})_{n \le -2d-2}$
and $(\F_{2n-1})_{n \le -d}$ by time-translations and by the slowing
procedure just introduced. And truncations, time-translations and slowing
procedure preserve the nature of the filtrations.
\end{proof}

\section{Improving on an example of Tsirelson}~\label{Tsirelson} 

In some non-published notes, Tsirelson gives a method to construct
an inhomogeneous Markov process $(X_n)_{n \le 0}$ such that the natural
filtration $(\fc^X_n)_{n \le 0}$ is easily proved to be non-standard,
although the tail $\sigma$-field $\F^X_{-\infty}$ is trivial and
the random variables $(X_{2n})_{n \le 0}$ are independent.

In Tsirelson's construction, each triple $(X_{2n-2},X_{2n-1},X_{2n})$ has
a particular structure  that we will explain soon. since the sequence
$(X_n)_{n \le 0}$ is obtained by
gluing the triples $(X_{2n-2},X_{2n-1},X_{2n})$ in a Markovian way, we
call {\it Tsirelson's bricks} these triples.

\subsection{The basic Tsirelson's brick}

Informally, the basic brick in Tsirelson's construction is a triple
of uniform random variables $X_{0},X_{1},X_{2}$ with values in some
finite sets $F_{0},F_{1},F_{2}$ such that for some $\alpha \in [0,1[$,
\begin{itemize}
\item the set $F_{2}$ is arbitrarily large, and the set $F_{0}$
is much larger;
\item the triple $(X_{0},X_{1},X_{2})$ is Markov;
\item the random variables $X_{0}$ and $X_{2}$ are independent;
\item any two different values of $X_{0}$ lead
to different values of $X_{2}$ with probability $\ge 1-\alpha$.
\end{itemize}
We now explain what the last requirement means.

Fix two distinct values
in $F_{0}$, namely $x'_0$ and $x''_0$. Choose randomly but not necessarily
independently $x'_1$ and $x''_1$ in $F_{1}$ according to the laws
$\lc(X_1|X_0=x'_0)$ and $\lc(X_1|X_0=x''_0)$. Then choose randomly but
not necessarily independently $x'_2$ and $x''_2$ in $F_{2}$ according to
the laws $\lc(X_2|X_1=x'_1)$ and $\lc(X_2|X_1=x''_1)$. Then
the values $x'_2$ and $x''_2$ must be
different with probability $\ge 1-\alpha$, whatever was the strategy
used to make the different choices.

More precisely, note $\rho_2$ the discrete metric on $F_2$: for all $x'_2$
and $x''_2$ in $F_2$,
\begin{eqnarray*}
\rho_2(x'_2,x''_2)=1 & \text{ if } x'_2 \ne x''_2,\\
\rho_2(x'_2,x''_2)=0 & \text{ if } x'_2 = x''_2.
\end{eqnarray*}
For all $x'_1$ and $x''_1$ in $F_1$, note $\rho_1(x'_1,x''_1)$
the Kantorovitch-Rubinstein
distance between the laws $\lc(X_2|X_1=x'_1)$ and $\lc(X_2|X_1=x''_1)$.
By definition,
$$\rho_1(x'_1,x''_1) = \inf\{\eef[\rho_2(X'_2,X''_2)]\ ; X'_2 \leadsto
\lc(X_2|X_1=x'_1), X''_2 \leadsto \lc(X_2|X_1=x''_1) \}.$$
Since $\rho_2$ is the discrete metric on $F_2$, $\rho_1(x'_1,x''_1)$
is actually the total variation distance between $\lc(X_2|X_1=x'_1)$ and
$\lc(X_2|X_1=x''_1)$.

By the same way, for all $x'_0$ and $x''_0$ in $F_0$, denote by
$\rho_0(x'_0,x''_0)$ the Kantorovitch-Rubinstein
distance between the laws $\lc(X_1|X_0=x'_0)$ and $\lc(X_1|X_0=x''_0)$.
The last requirement means that $\rho_0(x'_0,x''_0) \ge 1-\alpha$
when $x'_0 \ne x''_0$. This condition is used by Tsirelson to negate
Vershik's criterion.

Here is another formulation, which is closer to the
I-cosiness criterion recalled in section~\ref{annex}: for any non-anticipative
coupling of two copies $(X'_{0},X'_{1},X'_{2})$ and
$(X''_{0},X''_{1},X''_{2})$ of $(X_{0},X_{1},X_{2})$,
defined on some probability space $(\bar{\Omega},\bar{\ac},\bar{\P})$,
$$\bar{\P}[X'_{2} \ne X''_{2}|\sigma(X'_{0},X''_{0})] \ge 1-\alpha
\text{ on the event } [X'_{0} \ne X''_{0}].$$
Here, the
expression ``non-anticipative'' means that the filtrations generated by
the processes $X'$ and by $X''$ are immersed in the natural filtration
of $(X',X'')$. In particular, $X'_{1}$ and $X''_{0}$ are independent
conditionally on $X'_{0}$ (the couple $(X'_{0},X''_{0})$ gives no more
information on $X'_{1}$ than $X'_{0}$ does). Similarly, $X'_{2}$ and
$(X''_{0},X''_{1})$ are independent conditionally on $(X'_{0},X'_{1})$.
And the same holds when the roles of $X'$ and $X''$ are exchanged.

Let us give a formal definition.

\begin{defi}
Fix $\alpha \in ]0,1[$. Let $F_{0},F_{1},F_{2}$ be finite sets.
We will say that a triple $(Z_{0},Z_{1},Z_{2})$ of uniform random variables
with values in $F_{0},F_{1},F_{2}$ is a Tsirelson's $\alpha$-brick if
\begin{itemize}
\item the triple $(Z_{0},Z_1,Z_{2})$ is Markov.
\item $Z_{0}$ and $Z_{2}$ are independent.
\item for any non-anticipative
coupling of two copies $(X'_{0},X'_{1},X'_{2})$ and
$(X''_{0},X''_{1},X''_{2})$ of $(X_{0},X_{1},X_{2})$,
defined on some probability space $(\bar{\Omega},\bar{\ac},\bar{\P})$,
$$\bar{\P}[X'_{2} \ne X''_{2}|\sigma(X'_{0},X''_{0})] \ge 1-\alpha
\text{ on the event } [X'_{0} \ne X''_{0}].$$
\end{itemize}
\end{defi}

\subsection{Tsirelson's example of a brick}

Tsirelson gives an example of such a brick which is enlightening.

Let $p$ be a prime number, and $\zzf_p$ be the finite field with $p$ elements.
Note $F_{0}$ the set of all two-dimensional linear subspaces of
$(\zzf_p)^5$, $F_{1}$ the set of all one-dimensional affine subspaces of
$(\zzf_p)^5$ and $F_{2} = (\zzf_p)^5$. Then the size of $F_2$ is $|F_{2}| = p^5$ 
whereas
$$|F_{0}| = \frac{(p^5-1)(p^5-p)}{(p^2-1)(p^2-p)}
= (p^4+p^3+p^2+p+1)(p^2+1).$$
Indeed, the number of couples of independent vectors in $(\zzf_p)^5$ is
$(p^5-1)(p^5-p)$, but any linear plane in $(\zzf_p)^5$ can be generated by
$(p^2-1)(p^2-p)$ of these couples.

Tsirelson constructs a Markovian triple $(X_{0},X_{1},X_{2})$ as follows:
\begin{itemize}
\item choose uniformly $X_{0}$ in $F_{0}$ ;
\item given $X_{0}$, choose uniformly $X_{1}$ among the affine lines
whose direction are included in the linear plane $X_{0}$ ;
\item given $X_{0}$ and $X_{1}$, choose uniformly $X_{2}$ on the affine
line $X_{1}$.
\end{itemize}
One can check that $X_{2}$ is uniform on $F_{2}$, and independent of $X_{0}$.

Now, let $(X'_{0},X'_{1},X'_{2})$ and $(X''_{0},X''_{1},X''_{2})$
be any non-anticipative coupling of two copies of $(X_{0},X_{1},X_{2})$,
defined on some probability space $(\bar{\Omega},\bar{\ac},\bar{\P})$.
Then, conditionally on $(X'_{0},X''_{0},X'_{1},X''_{1})$,
the law of $X'_{2}$ is uniform on the line $X'_{1}$ and
the law of $X''_{2}$ is uniform on the $X''_{1}$. Since two distinct lines
have at most one common point, one has
$$\bar{\P}[X'_2 = X''_2|\sigma(X'_{0},X''_{0},X'_{1},X''_{1})] \le
\one_{[X'_{1} = X''_{1}]} + \frac{1}{p} \one_{[X'_{1} \ne X''_{1}]},$$
hence
$$\bar{\P}[X'_2 \ne X''_2|\sigma(X'_{0},X''_{0},X'_{1},X''_{1})] \ge
\frac{p-1}{p} \one_{[X'_{1} \ne X''_{1}]}.$$
Similarly, conditionally on $(X'_{0},X''_{0})$,
the law of $X'_{1}$ is uniform on the set of
all affine lines which are parallel to $X'_{0}$ and
the law of $X''_{1}$ is uniform on the set of
all affine lines which are parallel to $X''_{0}$.
But the affine lines $X'_{1}$ and $X''_{1}$ must have the same direction
to be equal. Since each linear plane in $(\zzf_p)^5$ contains $p+1$
linear lines whereas two distinct planes contain at most one common line,
$$\bar{\P}[X'_{1} = X''_{1}|\sigma(X'_{0},X''_{0})] \le
\one_{[X'_{0} = X''_{0}]} + \frac{1}{p+1} \one_{[X'_{0} \ne X''_{0}]},$$
hence
$$\bar{\P}[X'_{1} \ne X''_{1}|\sigma(X'_{0},X''_{0})] \ge
\frac{p}{p+1} \one_{[X'_{0} \ne X''_{0}]}.$$
Putting things together, one gets
\begin{eqnarray*}
\bar{\P}[X'_2 \ne X''_2|\sigma(X'_{0},X''_{0})]
&\ge& \frac{p-1}{p} \P[X'_{1} \ne X''_{1}|\sigma(X'_{0},X''_{0})] \\
&\ge& \frac{p-1}{p+1} \one_{[X'_{0} \ne X''_{0}]}.
\end{eqnarray*}
Hence, $(X_{0},X_{1},X_{2})$ is a Tsirelson's $\alpha$-brick with
$\alpha = 2/(p+1)$.

\subsection{Assembling bricks together}

The next step is to construct a non-homogeneous Markov process
$(X_n)_{n \le 0}$ such that for each $n \le 0$, the subprocess
$(X_{2n-2},X_{2n-1},X_{2n})$ is an Tsirelson's $\alpha_n$-brick,
where the $]0,1[$-valued sequence $(\alpha_n)_{n \le 0}$ fulfills
$$\sum_{n \le 0} \alpha_n < +\infty.$$
The next theorem achieves Tsirelson's construction. 

\begin{theo}~\label{Tsirelson's construction}
Let $(X_n)_{n \le 0}$ be a sequence of uniform random variables
with values in finite sets $(F_n)_{n \le 0}$ and $(\alpha_n)_{n \le 0}$
be an $]0,1[$-valued sequence such that the series $\sum_n \alpha_n$
converges.  Assume that
\begin{itemize}
\item the sets $F_{2n}$ are not singles,
\item $(X_n)_{n \le 0}$ is a non-homogeneous Markov process,
\item for each $n \le 0$, the subprocess
$(X_{2n-2},X_{2n-1},X_{2n})$ is a Tsirelson's $\alpha_n$-brick.
\end{itemize}
Then the natural filtration $\F^X$ is not standard.
Moreover, if the tail $\sigma$-field $\fc^X_{-\infty}$ is trivial,
then $|F_{2n}| \to +\infty$ as $n \to -\infty$.
\end{theo}

\begin{proof}[Proof of theorem~\ref{Tsirelson's construction}]
First, we show that $X_0$ does not fulfills the I-cosiness criterion
(see section~\ref{annex}). Indeed, set
$$c = \prod_{k \le 0}(1-\alpha_k) > 0$$
and consider any non-anticipative coupling $(X'_n)_{n \le 0}$ and
$(X''_n)_{n \le 0}$ of the process $(X_n)_{n \le 0}$,
defined on some probability space $(\bar{\Omega},\bar{\ac},\bar{\P})$.
By assumption, for every $n \le 0$,
$$\bar{\P}[X'_{2n} \ne X''_{2n}|\sigma(X'_{2n-2},X''_{2n-2})] \ge
(1-\alpha_n) \one_{[X'_{2n-2} \ne X''_{2n-2}]}.$$
By induction, for every $n \le 0$,
\begin{eqnarray*}
\bar{\P}[X'_0 \ne X''_0|\sigma(X'_{2n},X''_{2n})]
\ge \Big(\prod_{k=n+1}^0 (1-\alpha_k) \Big) \one_{[X'_{2n} \ne X''_{2n}]} 
\ge c \one_{[X'_{2n} \ne X''_{2n}]} 
\end{eqnarray*}
If, for some $N \le 0$, the $\sigma$-fields $\F^{X'}_{2N}$ and
$\F^{X''}_{2N}$ are independent, then
$$\bar{\P}[X'_0 \ne X''_0] \ge c \bar{\P}[X'_{2N} \ne X''_{2N}]
= c(1-|F_{2n}|^{-1}) \ge c/2.$$
Hence $\bar{\P}[X'_0 \ne X''_0]$ is bounded away from $0$, which negates
the I-cosiness criterion. The non-standardness of $\F^X$ follows.

The second part of the theorem directly follows from the next proposition,
applied to the sequence $(Y_n)_{n \le 0} = (X_{2n})_{n \le 0}$.
\end{proof}


\begin{prop}~\label{necessity of large sets}
Let $(\gamma_n)_{n \le 0}$ be a sequence of positive constants such that
$$\prod_{n \le 0}\gamma_n > 0.$$
Let $(Y_n)_{n \le 0}$ be a family of random variables which are uniformly
distributed on finite sets $(E_n)_{n \le 0}$. Let $(Y'_n)_{n \le 0}$ and
$(Y''_n)_{n \le 0}$ be independent copies of the process $(Y_n)_{n \le 0}$,
defined on some probability space $(\bar{\Omega},\bar{\ac},\bar{\P})$.
Assume that $\fc^Y_{-\infty}$ is trivial and that for every $n \le 0$,
$$\bar{\P}[Y'_{n} \ne Y''_{n}|\sigma(Y'_{n-1},Y''_{n-1})] \ge \gamma_n
\one_{[Y'_{n-1} \ne Y''_{n-1}]}.$$
Then $|E_n| \to 1$ or $|E_n| \to +\infty$ as $n \to -\infty$.
\end{prop}

\begin{proof}[Proof of proposition~\ref{necessity of large sets}]
By the independence of $(Y'_n)_{n \le 0}$ and $(Y''_n)_{n \le 0}$, the following
exchange properties apply (see~\cite{Weizsacker})
\begin{eqnarray*}
\bigcap_{m \le 0} \bigcap_{n \le 0} \left( \fc^{Y'}_m \vee \fc^{Y''}_n \right)
&=& \bigcap_{m \le 0}
\left( \fc^{Y'}_m \vee \Big( \bigcap_{n \le 0} \fc^{Y''}_n \Big) \right) \\
&=& \bigcap_{m \le 0}
\left( \fc^{Y'}_m \vee \fc^{Y''}_{-\infty} \right) \\
&=& \left( \bigcap_{m \le 0} \fc^{Y'}_m \right) \vee \fc^{Y''}_{-\infty} \\
&=& \fc^{Y'}_{-\infty} \vee \fc^{Y''}_{-\infty}.
\end{eqnarray*}
Using that $\fc^{Y'}_m \vee \fc^{Y''}_n$ is non-decreasing with respect to
$m$ and $n$, one gets 
$$(\fc^{Y'} \vee \fc^{Y''})_{-\infty} =
\bigcap_{n \le 0} \left( \fc^{Y'}_n \vee \fc^{Y''}_n \right) =
\bigcap_{m \le 0} \bigcap_{n \le 0} \left( \fc^{Y'}_m \vee \fc^{Y''}_n \right).$$
Hence the tail $\sigma$-field $(\fc^{Y'} \vee \fc^{Y''})_{-\infty}$ is trivial.
Thus the asymptotic event
$$\liminf_{n \to -\infty} [Y'_n \ne Y''_n]$$
has probability $0$ or $1$.

But a recursion shows that for every $n \le 0$
$$\bar{\P} \Big( \bigcap_{n \le k \le 0} [Y'_{k} \ne Y''_{k}] \Big|
\sigma(Y'_{n},Y''_{n}) \Big)
\ge \Big( \prod_{n+1 \le k \le 0} \gamma_k \Big) \one_{[Y'_{n} \ne Y''_{n}]}.$$
By taking expectations,
$$\bar{\P} \Big( \bigcap_{n \le k \le 0} [Y'_{k} \ne Y''_{k}] \Big) \ge
\big(1-|E_n|^{-1} \big) \prod_{n+1 \le k \le 0} \gamma_k.$$

If $|E_n| \ge 2$ for infinitely many $n \le 0$, then
$$\bar{\P} \Big( \bigcap_{k \le 0} [Y'_{k} \ne Y''_{k}] \Big) \ge
\frac{1}{2} \prod_{k \le 0} \gamma_k > 0.$$
Thus $|E_n| \ge 2$ for every $n \le 0$ and
$$\bar{\P}(\liminf_{n \to -\infty} [Y'_n \ne Y''_n])=1.$$
But by Fatou's lemma,
$$\bar{\P}(\liminf_{n \to -\infty} [Y'_n \ne Y''_n]) \le
\liminf_{n \to -\infty} \bar{\P}[Y'_n \ne Y''_n].$$
Hence $1-|E_n|^{-1} = \bar{\P}[Y'_n \ne Y''_n] \to 1$ thus
$|E_n| \to +\infty$ as $n \to -\infty$.
\end{proof}

\subsection{Choosing the size of the sets $F_n$}

The last theorem explains the necessity to have bricks $(Z_{0},Z_{1},Z_{2})$
such that the set $F_{2}$ of all possible values of $Z_2$ is arbitrarily
large, and the set $F_{0}$  of all possible values of $Z_0$ is much larger.
In Tsirelson's example, the size of $F_{2}$ is $p^5$ where $p$ is a prime
number, whereas the size of $F_{0}$ is $(p^4+p^3+p^2+p+1)(p^2+1)$.

Such bricks provided cannot be glued together since the size of $F_{2}$
is not a power of a prime number: it has at least two prime divisors since 
the greatest common divisor of
$p^4+p^3+p^2+p+1$ and $p^2+1$ is $1$. Replacing $\zzf_p$ by a more general
finite field would not change anything since the size of any finite field
is necessarily a power of a prime number. Fortunately, a slight modification
solve this problem.

A first way to solve the problem is to choose a prime number $q$ such that
$q^5$ is slightly smaller than $(p^4+p^3+p^2+p+1)(p^2+1)$ and to call
$F_{0}$ a subset with size $q^5$ of all two-dimensional linear
subspaces of $(\zzf_p)^5$. After this modification, the law of
$Z_{1}$ (a random line choose uniformly along the affine lines which
are parallel to the linear plane $Z_{0}$) will no longer be an uniform law,
but the law of $Z_{1}$ plays no particular role in the construction.

A second solution is to replace the affine {\it lines} by the affine
{\it planes} in the definition of $Z_{1}$ and $F_{1}$.
In this last solution, $Z_{0}$ is a deterministic function of
$Z_{1}$ (namely, the vector plane is the direction of the affine plane)
and $Z_{1}$ is a deterministic function of $(Z_{0},Z_{2})$
(namely, $Z_{1}$ is the only affine plane
which is parallel to $Z_{0}$ and contains $Z_{2}$). These two additional
properties have many advantages. First, the construction and the proofs are
even simpler. Next, we will use them to get stronger results.

From now on, we will consider only bricks having these two additional
properties.

\subsection{Strong bricks}

Let us give a rigorous definition.

\begin{defi}
Fix $\alpha \in ]0,1[$ and two positive integers $r_{1},r_{2}$.
Let $F_{0},F_{1},F_{2}$ be finite sets.
We will say that a triple $(Z_{0},Z_{1},Z_{2})$ of uniform random variables
with values in $F_{0},F_{1},F_{2}$ is a strong $(r_{1},r_{2})$-adic $\alpha$-brick
if
\begin{itemize}
\item $Z_{0}$ and $Z_{2}$ are independent.
\item $Z_{1}$ is a deterministic function of $(Z_{0},Z_{2})$;
\item $Z_{0}$ is a deterministic function of $Z_{1}$;
\item the conditional law of $Z_{1}$ given $Z_{0}$
is uniform on some finite random set of size $r_{1}$;
\item the conditional law of $Z_{2}$ given $Z_{1}$
is uniform on some finite random set of size $r_{2}$;
\item for every distinct elements $z'_{1}$ and $z''_{1}$ in $F_{1}$,
\begin{equation}~\label{bad coupling}
\sum_{z \in F_{2}} \min \big(\P[Z_{2}=z|Z_{1}=z'_{1}],\P[Z_{2}=z|Z_{1}=z''_{1}] \big)
\le \alpha.
\end{equation}
\end{itemize}
\end{defi}

The next lemma shows that the definition of strong bricks is more restrictive
that the definition of Tsirelson's bricks.

\begin{lemm}~\label{strong bricks are bricks}
If $(Z_{0},Z_{1},Z_{2})$ is a strong $\alpha$-brick, then
for any non-anticipative coupling $(Z'_{0},Z'_{1},Z'_{2})$ and
$(Z''_{0},Z''_{1},Z''_{2})$ of $(Z_{0},Z_{1},Z_{2})$, defined on some
probability space $(\bar{\Omega},\bar{\ac},\bar{\P})$,
$$\bar{\P}[Z'_2 \ne Z''_2|\sigma(Z'_{1},Z''_{1})] \ge
(1-\alpha) \one_{[Z'_{1} \ne Z''_{1}]} \ge
(1-\alpha) \one_{[Z'_{0} \ne Z''_{0}]}.$$
Thus, $(Z_{0},Z_{1},Z_{2})$ is a Tsirelson's $\alpha$-brick
\end{lemm}

\begin{proof}[Proof of lemma~\ref{strong bricks are bricks}]
The triple $(Z_{0},Z_{1},Z_{2})$ is Markov since $Z_{0}$ is a function of $Z_{1}$.

Now, let $(Z'_{0},Z'_{1},Z'_{2})$ and $(Z''_{0},Z''_{1},Z''_{2})$
be any non-anticipative coupling of $(Z_{0},Z_{1},Z_{2})$, defined on some
probability space $(\bar{\Omega},\bar{\ac},\bar{\P})$.
Set $\gc = \sigma(Z'_{0},Z'_{1},Z''_{0},Z''_{1})$.
By the non-anticipative and the Markov properties,
$$\lc(Z'_{2}|\gc) = \lc(Z'_{2}|\sigma(Z'_{0},Z'_{1}))
= \lc(Z'_{2}|\sigma(Z'_{1}))$$
and the same holds with $Z''$.

Thus for any distinct values $z',z''$ in $F_{1}$, one has, on the event
$[Z'_{1}=z'\ ;\ Z''_{1}=z'']$,
\begin{eqnarray*}
\P[Z'_{2} = Z''_{2}|\gc]
&=& \sum_{z \in F_{2}} \P[Z'_{2} = z\ ;\ Z''_{2}=z|\gc]\\
&\le& \sum_{z \in F_{2}} \P[Z'_{2n} = z|\gc] \wedge \P[Z''_{2}=z|\gc]\\
&=& \sum_{z \in F_{2}} \P[Z'_{2} = z|Z'_{1}=z'] \wedge \P[Z''_{2}=z|Z''_{1}=z'']\\
&=& \sum_{z \in F_{2}} \P[Z_{2} = z|Z_{1}=z'] \wedge \P[Z_{2}=z|Z_{1}=z'']\\
&\le& \alpha.
\end{eqnarray*}
Hence
$$\P[Z'_{2}  = Z''_{2}|\gc]
\le \alpha \one_{[Z'_{1} \ne Z''_{1}]} + \one_{[Z'_{1} = Z''_{1}]}.$$
Taking complements, one gets
$$\P[Z'_{2} \ne Z''_{2}|\gc]
\ge (1-\alpha) \one_{[Z'_{1} \ne Z''_{1}]}.$$
The last inequality follows from the inclusion
$[Z'_{0} \ne Z''_{0}] \subset [Z'_{1} \ne Z''_{1}]$.
\end{proof}

As we now see, the definition of a strong brick provides constraints
on the size of the sets $F_{0},F_{1},F_{2}$.

\begin{lemm}~\label{constraints} {\bf (Properties of bricks)}
Fix $\alpha \in ]0,1[$ and two positive integers $r_{1},r_{2}$.
Let $F_{0},F_{1},F_{2}$ be finite sets. Assume the existence of
a triple $(Z_{0},Z_{1},Z_{2})$ of uniform random variables with
values in $F_{0},F_{1},F_{2}$ such that $(Z_{0},Z_{1},Z_{2})$
is a $(r_{1},r_{2})$-adic $\alpha$-brick. Let $f : F_1 \to F_0$
and $g : F_0 \times F_2 \to F_1$ be the maps such that
$f(Z_1)=Z_0$ and $g(Z_0,Z_2)=Z_1$. Then:
\begin{enumerate}
\item the map $f$ is $r_1$ to one and the map $g$ is $r_2$ to one.
More precisely, for every $z_1 \in F_1$,
$g^{-1}(\{z_1\}) = \{f(z_1)\} \times S(z_1)$ where $S(z_1)$ is a subset
of $F_2$ of size $r_2$.
\item for every $z_1 \in F_{1}$, the law of $Z_2$
conditionally on $Z_1=z_1$ is uniform on $S(z_1)$.
\item for each $z_0 \in F_0$, the subsets $S(z_1)$ for $z_1 \in f^{-1}(\{z_0\})$
form a partition of $F_2$ in $r_1$ blocks.
\item $|F_1| = r_1|F_0|$, $|F_0 \times F_2| = r_2|F_1|$
and $|F_2| = r_{1}r_{2}$.
\item for every distinct elements $z'_1$ and $z''_1$ in $F_{1}$,
$|S(z'_1) \cap S(z''_1)| \le \alpha r_2$.
\item if $|F_0| \ge 2$, then $r_2 \ge 1/\alpha$.
\end{enumerate}
\end{lemm}

\begin{proof}[Proof of lemma~\ref{constraints}]
By hypothesis, for every $(z_0,z_1,z_2) \in F_0 \times F_1 \times F_2$,
$$\P[Z_0=z_0\ ;\ Z_1=z_1] = \frac{1}{|F_1|} \one_{[z_0=f(z_1)]}.$$
Hence
$$\P[Z_1=z_1|Z_0=z_0]
= \frac{1}{|f^{-1}(\{z_0\})|} \one_{f^{-1}(\{z_0\})}(z_1),$$
which shows that $|f^{-1}(\{z_0\})| = r_1$.

By the same way,
$$\P[Z_0=z_0\ ;\ Z_1=z_1\ ;\ Z_2=z_2] = \frac{1}{|F_0 \times F_2|}
\one_{[z_1=g(z_0,z_2)]}.$$
Hence
$$\P[Z_0=z_0\ ;\ Z_2=z_2|Z_1=z_1] =
\frac{1}{|g^{-1}(\{z_1\})|} \one_{g^{-1}(\{z_1\})}(z_0,z_2),$$
which shows that $|g^{-1}(\{z_1\})| = r_2$.

Since $(Z_0,Z_2)$ is uniform on $F_0 \times F_2$,
the equalities $Z_0 = f(Z_1)$ and $Z_1 = g(Z_0,Z_2)$ shows that
$z_0 = f(g(z_0,z_2))$ for every $(z_0,z_2) \in F_0 \times F_2$. Hence,
for every $z_1 \in F_1$, if $(z_0,z_2) \in g^{-1}(\{z_1\})$ then $z_0 = f(z_1)$.
This shows that
$g^{-1}(\{z_1\}) = \{f(z_1)\} \times S(z_1)$ where $S(z_1)$ is some subset
of $F_2$.

Thus, for every $(z_1,z_2) \in F_1 \times F_2$,
$$\P[Z_2=z_2|Z_1=z_1] = \P[Z_0=f(z_1)\ ;\ Z_2=z_2|Z_1=z_1] =
\frac{1}{|S(z_1)|} \one_{S(z_1)}(z_2).$$
Hence the law of $Z_2$ conditionally on $Z_1=z_1$ is uniform on $S(z_1)$
which has size $r_2$. This completes the proof of the first two points.

The third and fourth points follow.

Fix two distinct elements $z'_1$ and $z''_1$ in $F_{1}$.
Then for every $z \in F_{2}$,
$$\min \big(\P[Z_{2}=z|Z_{1}=z'_1],\P[Z_{2}=z|Z_{1}=z''_1]
= \frac{1}{r_2} \min (\one_{S(z'_1)}(z),\one_{S(z''_1)}(z)).$$
Summing over $z$ and using the inequality~\ref{bad coupling},
one gets
$$|S(z'_1) \cap S(z''_1)| \le \alpha r_2,$$
which is the fifth point.

If $|F_0| \ge 2$, then one can choose two distinct elements
$z'_0$ and $z''_0$ in $F_0$. Let $z_2 \in F_2$, $z'_1 = g(z'_0,z_2)$ and
$z''_1 = g(z'_0,z_2)$. Then $z'_1$ and $z''_1$ are distinct elements
in $F_{1}$ since $f(z'_1)=z'_0$ and $f(z''_1)=z''_0$ are distinct.
But $z_2$ belongs to $S(z'_1)$ since
$$\P[Z_1=z'_1|Z_2=z_2] = \P[Z_0=z'_0|Z_2=z_2]= |F_0|^{-1},$$
and $z_2$ also belongs to $S(z''_1)$.
Hence $1 \le |S(z'_1) \cap S(z''_1)| \le \alpha r_2$.
which shows the sixth point.
\end{proof}

\subsection{Getting bricks}

The next lemma provides a general method to get bricks.

\begin{lemm}~\label{general method} {\bf (Method to get bricks)}

Fix $\alpha \in ]0,1[$ and two positive integers $r_{1},r_{2}$.

Let $F_{0},F_{2}$ be finite sets such that $F_{2}$ has size $r_{1}r_{2}$.

Let $Z_{0}$ and $Z_{2}$ be independent random variables, uniformly
distributed in $F_{0}$ and $F_{2}$.

Let $(\Pi_z)_{z \in F_{0}}$ be a family of partitions of $F_{2}$
indexed by $F_{0}$ such that
\begin{itemize}
\item each partition $\Pi_z$ has $r_{1}$ blocks
$S_{z,1},\ldots,S_{z,r_{1}}$;
\item each block has $r_{2}$ elements.
\item for any distinct $(z',i')$ and $(z'',i'')$ in $F_{0} \times \odc 1,r_{1} \fdc$,
$|S_{z',i'} \cap S_{z'',i''}| \le \alpha r_{2}$. 
\end{itemize}
(This ``transversality condition'' forces the partitions to be all different
and says that two blocks chosen in any two different partitions have a small
intersection.)

Define a random variable with values in $F_{1} = F_{0} \times \odc 1,r_{1} \fdc$ by
$Z_{1} = (Z_{0},J)$, where $J$ is the index of the only block of
$\Pi_{Z_{0}}$ which contains $Z_{2}$ (that is to say $Z_{2} \in S_{Z_{0},J})$.

Then $(Z_{0},Z_{1},Z_{2})$ is a $(r_{1},r_{2})$-adic $\alpha$-brick.
\end{lemm}

\begin{proof}[Proof of lemma~\ref{general method}]
The first statement is obvious.

For every $z_{0} \in F_{0}$, $j \in \odc 1,r_{1} \fdc$ and $z_2 \in F_{2}$,
\begin{eqnarray*}
\P[Z_{0}=z_{0}\ ;\ J=j\ ;\ Z_{2}=z_2]
&=& \one_{[z_2 \in S_{z_{0},j}]}\ \P[Z_{0}=z_{0}\ ;\ Z_{2}=z_2]\\
&=& \one_{[z_2 \in S_{z_{0},j}]}\ \times \frac{1}{|F_{0}|}\
\times \frac{1}{r_{1}r_{2}}.
\end{eqnarray*}
Summing over $z_2$ yields
$$\P[Z_{0}=z_{0}\ ;\ J=j] = \frac{1}{|F_{0}|}\ \times \frac{1}{r_{1}}.$$
By division, one gets
$$\P[Z_{2}=z_2\ |\ Z_{0}=z_{0}\ ;\ J=j] = \one_{[z_2 \in S_{z_{0},j}]}
\ \times \frac{1}{r_{1}}.$$
The last two equalities show that $J$ is independent of $Z_{0}$ and
uniform on $\odc 1,r_{1} \fdc$, and that given $(Z_{0},J)$, $Z_{2}$ is uniform
on the block $S_{Z_{0},J}$. This proves the third and the fourth statement.

Let $z'_1$ and $z''_1$ be distinct elements in $F_{1}$. Conditionally on
$[Z_{1}=z'_1]$, the law $Z_{2}$ is uniform on the block $S_{z'_1}$.
Conditionally on $[Z_{1}=z''_1]$, the law $Z_{2}$ is uniform on the block
$S_{z''_1}$.
Thus
$$\sum_{z \in F_{2}} \P[Z_{2}=z|Z_{1}=z'_1] \wedge \P[Z_{2}=z|Z_{1}=z''_1]
= \sum_{z \in S_{z'_1} \cap S_{z''_1}} \frac{1}{r_2} \le \alpha.$$
The last statement follows.
\end{proof}

\subsection{Examples of bricks}

Algebra helps us to construct many partitions on a given set such that
each partition has a fix number of blocks, each block has a fix number
of elements and any two blocks chosen in any two different partitions
have a small intersection.

Let $q$ be any power of a prime number. Let $K$ be the field with $q$
elements, and $L$ the field with $q^2$ elements. Since $L$ is a quadratic
extension of $K$, $L$ is isomorphic to $K^2$ as a vector space on $K$.
Actually, one only needs to have a bijection between $K^2$ and $L$.

{\bf First example}

We set $r_{1}=r_{2}=q^4$,
$F_{0} = L^8$ (identified with the set $\mc_4(K)$ of all $4 \times 4$
matrices with entries in $K$) and
$F_{2} = K^8$ identified with $K^4 \times K^4$.

To each matrix $A \in \mc_4(K)$, one can associate the partition
of $K^8$ given by all four-dimensional affine subspaces of $K^8$
with equations $y=Ax+b$ where $b$ ranges over $K^4$. Each of these
subspaces has size $q^4$. But two subspaces of
equations $y=A'x+b'$ and $y=A''x+b''$ intersect in at most $q^3$ points
(a three dimensional affine subspace) when $A' \ne A''$.
Hence these partitions provide a $(q^4,q^4)$-adic $1/q$-brick.


{\bf Second example}

We set $r_{1}=r_{2}=q$,
$F_{0} = L^2$ (identified with $K^4$) and
$F_{2} = K^2$.

To each quadruple $(a,b,c,d) \in K^4$, one can associate the partition
of $K^2$ given by the $q$ graphs of equations $y=ax^4+bx^3+cx^3+dx+e$
where $e$ ranges over $K$. Each of these graphs has size $q$. But two
graphs with different $(a,b,c,d,e) \in K^4$ intersect in at most $4$ points.
Hence, if $p \ge 5$ these partitions provide a $(q,q)$-adic $4/q$-brick.

{\bf Gluing bricks together}

In both exemples above, the family of partitions provides bricks
which can be glued as follows. Let $q$ be any power of a prime number.
For each $n \le 0$, call $K_n$ the field with $q_n=q^{2^{|n|}}$ elements.
Set
$$\forall n \le 0,\ F_{2n} = K_n^8,\ r_{2n-1} = r_{2n} = q_n^4,
\alpha_n = 1/q_n \text{ and }
F_{2n-1} = F_{2n-2} \times \odc 1,r_{2n-1} \fdc$$
or
$$\forall n \le 0,\ F_{2n} = K_n^2,\ r_{2n-1} = r_{2n} = q_n,
\alpha_n = 4/q_n\text{ and }
F_{2n-1} = F_{2n-2} \times \odc 1,r_{2n-1} \fdc.$$
Start with a sequence of independent random variables
$(Z_{2n})_{n \le 0}$. For each $n \le 0$, consider the partitions of
$F_{2n}$ provided by the first or the second example and define
$Z_{2n-1}$ from $Z_{2n-2}$ and $Z_{2n}$ as in lemma~\ref{general method}.
By construction, $(Z_{2n},Z_{2n-1},Z_{2n})$ is an $(r_{2n-1},r_{2n})$-adic
$\alpha_n$-brick.

The next theorem shows that the process $(Z_n)_{n \le 0}$ thus defined
provides an example which proves the existence stated in
theorem~\ref{existence theorem}.

\subsection{Proof of theorem~\ref{existence theorem}}

Theorem~\ref{existence theorem} directly follows from the construction above
and from the theorem below.

\begin{theo}~\label{simple construction}
Let $(\alpha_n)_{n \le 0}$ be a sequence of reals in $]0,1[$ such that the
series $\sum_n \alpha_n$ converges.
Let $(Z_n)_{n \le 0}$ be any sequence of random variables taking values in some
finite sets $(F_n)_{n \le 0}$ of size $\ge 2$. Assume that
\begin{itemize}
\item the random variables $(Z_{2n})_{n \le 0}$ are independent;
\item for each $n \le 0$, $(Z_{2n-2},Z_{2n-1},Z_{2n})$ is an 
$(r_{2n-1},r_{2n})$-adic $\alpha_n$-brick.
\end{itemize}
Then
\begin{itemize}
\item $(Z_n)_{n \le 0}$ is a Markov process which generates a
$(r_n)$-adic filtration;
\item for every infinite subset $D$ of $\zzf_-$,
$(\fc^Z_{n})_{n \in D}$ is standard if and only if $2n-1 \notin D$
for infinitely many $n \le 0$.
\end{itemize}
In particular, the filtration
$(\fc^Z_{2n-1})_{n \le 0}$ is at the threshold of standardness.
\end{theo}


\begin{proof}[Proof of theorem~\ref{simple construction}]
We now prove the statements.

{\bf Proof that $(Z_n)_{n \le 0}$ is a Markov process and generates
a $(r_n)$-adic filtration}

First, note that the filtration $(\fc^Z_{2n})_{n \le 0}$ is generated
by the independent random variables $(Z_{2n})_{n \le 0}$ since
for every $n \le 0$, $Z_{2n-1}$ is a deterministic function of
$(Z_{2n-2},Z_{2n})$. Hence, for every $n \le 0$,
$$\fc^Z_{2n-2} = \sigma(Z_{2n-2}) \vee \fc^Z_{2n-4},$$
Moreover, since $Z_{2n-2}$ is a deterministic function of $Z_{2n-1}$,
$$\fc^Z_{2n-1} = \sigma(Z_{2n-1}) \vee \fc^Z_{2n-2}= \sigma(Z_{2n-1}) 
\vee \fc^Z_{2n-4}.$$
By independence of $(Z_{-2n-2},Z_{2n-1},Z_{-2n})$ and $\fc^Z_{2n-4}$, we get
$$\lc(Z_{2n-1}|\fc^Z_{2n-2}) = \lc(Z_{2n-1}|\sigma(Z_{2n-2})),$$
$$\lc(Z_{2n}|\fc^Z_{2n-1}) = \lc(Z_{2n}|\sigma(Z_{2n-1})).$$
The Markov property follows. But for every $n \le 0$, 
$(Z_{2n-2},Z_{2n-1},Z_{2n})$ is an $(r_{2n-1},r_{2n})$-adic
$\alpha_n$-brick. The $(r_n)$-adic character of $\fc^Z$ follows.

{\bf Proof that $(\fc^Z_n)_{n \in D}$ is not standard when $D$ contains 
all but finitely many odd negative integers}

First, we show that $(\fc^Z_{2n-1})_{n \le 0}$ is not standard. To do this,
we check that
the random variable $Z_{-1}$ does not satisfy the I-cosiness criterion.
Note that $(\fc^Z_{2n-1})_{n \le 0}$ is the natural filtration of
$(Z_{2n-1})_{n \le 0}$ only since for every $n \le 0$,
$Z_{2n-2}$ is some deterministic function $f_n$ of $Z_{2n-1}$.

Let $(Z'_{2n-1})_{n \le 0}$ and $(Z'_{2n-1})_{n \le 0}$ be two copies of the
process $(Z_{2n-1})_{n \le 0}$, defined on some probability space
$(\bar{\Omega},\bar{\ac},\bar{\P)}$.
Set  $Z'_{2n-2} = f_n(Z'_{2n-1})$ and $Z''_{2n-2} = f_n(Z''_{2n-1})$
for every $n \le 0$.
Then $(Z'_{n})_{n \le 0}$ and $(Z''_{n})_{n \le 0}$ are copies of the
process $(Z_n)_{n \le 0}$. Moreover, $(\fc^{Z'}_{2n-1})_{n \le 0}$ and
$(\fc^{Z''}_{2n-1})_{n \le 0}$ are the natural filtrations of
$(Z'_{2n-1})_{n \le 0}$ and $(Z''_{2n-1})_{n \le 0}$.

Assume that these filtrations are immersed in some filtration
$(\gc_{2n-1})_{n \le 0}$.
Then, for every $n \le -1$,
$$\lc(Z'_{2n+1}|\gc_{2n-1}) = \lc(Z'_{2n+1}|\fc^{Z'}_{2n-1})
= \lc(Z'_{2n+1}|\sigma(Z'_{2n-1})),$$
and since $Z'_{2n}$ is a deterministic function of $Z'_{2n+1}$,
$$\lc(Z'_{2n}|\gc_{2n-1}) 
= \lc(Z'_{2n}|\sigma(Z'_{2n-1})).$$
The same holds with the process $Z''$.

For any distinct values $z',z''$ in $F_{2n-1}$, one has on the event
$[Z'_{2n-1}=z'\ ;\ Z''_{2n-1}=z'']$,
\begin{eqnarray*}
\P[Z'_{2n} = Z''_{2n}|\gc_{2n-1}]
&=& \sum_{z \in F_{2n}} \P[Z'_{2n} = z\ ;\ Z''_{2n}=z|\gc_{2n-1}]\\
&\le& \sum_{z \in F_{2n}} \P[Z'_{2n} = z|\gc_{2n-1}] \wedge \P[Z''_{2n}=z|\gc_{2n-1}]\\
&=& \sum_{z \in F_{2n}} \P[Z'_{2n} = z|Z'_{2n-1}=z'] \wedge
\P[Z''_{2n}=z|Z''_{2n-1}=z'']\\
&=& \sum_{z \in F_{2n}} \P[Z_{2n} = z|Z_{2n-1}=z'] \wedge \P[Z_{2n}=z|Z_{2n-1}=z'']\\
&\le& \alpha_n.
\end{eqnarray*}
Hence, since $[Z'_{2n+1} = Z''_{2n+1}] \subset [Z'_{2n} = Z''_{2n}]$,
\begin{eqnarray*}
\bar{\P}[Z'_{2n+1} = Z''_{2n+1}|\gc_{2n-1}]
&\le& \bar{\P}[Z'_{2n} = Z''_{2n}|\gc_{2n-1}]\\
&\le& \alpha_n\one_{[Z'_{2n-1} \ne Z''_{2n-1}]} + \one_{[Z'_{2n-1} = Z''_{2n-1}]}.
\end{eqnarray*}
Taking the complements, one gets
$$\bar{\P}[Z'_{2n+1} \ne Z''_{2n+1}|\gc_{2n-1}] \ge
(1-\alpha_n) \one_{[Z'_{2n-1} \ne Z''_{2n-1}]}.$$
A simple recursion yields
$$\P[Z'_{-1} \ne Z''_{-1}|\gc_{2n-1}] \
\ge \prod_{n \le k \le -1} (1-\alpha_k)\ \one_{[Z'_{2n-1} \ne Z''_{2n-1}]}.$$
Taking the expectations, one gets
$$\P[Z'_{-1} \ne Z''_{-1}] \
\ge \prod_{n \le k \le -1} (1-\alpha_k)\ \P[Z'_{2n-1} \ne Z''_{2n-1}].$$
Assume now that that for some $N > -\infty$, the $\sigma$-fields
$\fc'_{2N-1}$ and $\fc''_{2N-1}$ are independent. Then for every $n \le N$,
$$\P[Z'_{2n-1} \ne Z''_{2n-1}] = 1-\frac{1}{|F_{2n-1}|} \ge \frac{1}{2},$$
since $Z'_{2n-1}$ and $Z''_{2n-1}$ are independent and uniform on $F_{2n-1}$.
Going to the limit yields
$$\P[Z'_{-1} \ne Z''_{-1}] \ge \frac{1}{2} \prod_{k \le -1} (1-\alpha_k) > 0,$$
which shows that $Z_{-1}$ does not satisfy the I-cosiness criterion.

Thus $(\fc^Z_{2n-1})_{n \le 0}$ is not standard. Thus, if $D$ is any subset 
of $\zzf_-$ which contains all odd negative integers, 
the filtration $(\fc^Z_n)_{n \in D}$ is not standard (since standardness 
is preserved by extraction). This conclusion still holds
when $D$ contains all but finitely many odd negative integers 
(since standardness is an asymptotic property). 

{\bf Proof that $(\fc^Z_n)_{n \in D}$ is standard when $D$ skips
infinitely many odd negative integers}

Since standardness is preserved by extraction, one only needs to consider
the case where $D$ contains all even non-positive numbers. In this case,
the filtration $(\fc^Z_n)_{n \in D}$ is generated by $(Z_n)_{n \in D}$ only. 
Indeed, if $n$ is any integer in $\zzf_- \setminus D$, then $n$ is odd, 
hence $n-1 \in D$,
$n+1 \in D$ and $Z_n$ is a function of $(Z_{n-1},Z_{n+1})$.

For each $n \le 0$, the conditional law $\lc(Z_n|\fc^Z_{n-1})=\lc(Z_n|Z_{n-1})$
is (almost surely) uniform on some random subset of $F_n$ with $r_n$ elements.
By fixing a total order on the set $F_n$, one can construct an uniform random
variable uniform $U_n$ on $\odc 1,r_n \fdc$, independent of $\fc^Z_{n-1}$, 
such that $Z_n$ is a function of $Z_{n-1}$ and $U_n$.
Set $Y_n = Z_n$ if $n-1 \in D$ (which may happen only for even $n$)
and $Y_n = U_n$ otherwise. Then $Y_n$ is $\fc^Z_n$-measurable. This shows
that $\fc^Y_n \subset \fc^Z_n$ for every $n \in D$.

Let us prove the reverse
inclusion. Fix $n \in D$, and call $m \le n$ the integer
such that $m-1 \notin D$ but $k \in D$ for all $k \in \odc m,n \fdc$.
Then $Z_n$ is $\fc^Y_n$-measurable as a function of
$Y_m=Z_m,Y_{m+1}=U_{m+1},\ldots,Y_n=U_n$.

Last, for every $n \in D$, $Y_n$ is independent of
$\fc^Z_{n-1}$ if $n-1 \in D$ and $Y_n$ is independent of
$\fc^Z_{n-2}$ otherwise. This shows the independence of
the random variables $(Y_n)_{n \in D}$. Hence the filtration
$(\fc^Z_n)_{n \in D}$ is of product type, which completes the proof.
\end{proof}


\section{Annex: some basic facts on standardness}~\label{annex}

We summarize here the main definitions and results used in this paper.
A complete exposition can be found in~\cite{Emery-Schachermayer}.

Recall that we work with filtrations indexed by the non-positive integers
on a probability space $(\Omega,\ac,\P)$, and that all the
sub-$\sigma$-fields of $\ac$ that we consider here are assumed to be
complete and essentially separable with respect to $\P$.

Most of the time, the probability measure $\P$ is not explicitly mentioned
when we deal with filtrations. Yet, it actually plays an important
role and the true object of study are filtered probability spaces
$(\Omega,\ac,\P,(\fc_n)_{n \le 0})$.

\subsection{Isomorphisms of filtered probability spaces}

The definition of isomorphism is not as simple as one could expect.

Let $\F = (\F_n)_{n \le 0}$ and $\F' = (\F_n')_{n \le 0}$ be
filtrations on $(\Omega,\mathcal{A},\P)$ and $(\Omega',\mathcal{A}',\P')$.

\begin{defi}
An isomorphism of filtered probability spaces from
$(\Omega,\mathcal{A},\P,\F)$ into $(\Omega',\mathcal{A}',\P',\F')$
is a bijective (linear) application from the space
${\bf L}^0(\Omega,\F_{\infty},\P)$ of the real random
variables on $(\Omega, \F_{\infty}, \P)$
into ${\bf L}^0(\Omega',\F_{\infty}',\P')$
which preserves the laws of the random variables,
commutes with Borelian applications, and sends $\F$ on $\F'$.
\end{defi}

By definition, saying that an isomorphism $\Psi$ sends $\F$ on $\F'$
means that for every $n \le 0$, the random variables $\Psi(X)$ for $X \in 
{\bf L}^0(\Omega,\F_{n},\P)$ generate $\fc'_{n}$. Saying that $\Psi$ 
commutes with Borelian 
applications means that for every sequence $(X_n)_{n \ge 1}$ of real 
random variables on $(\Omega,\mathcal{A},\P)$, and every Borelian 
application $F : \R^\infty \to \R$, 
$$\Psi \big( F \circ (X_n)_{n \ge 1} \Big) = F \circ (\Psi(X_n))_{n \ge 1}.$$
In particular, this equality holds when $F$ is given by
$F((x_n)_{n \ge 1}) = \alpha_1 x_1 + \alpha_2 x_2$ with $(\alpha_1,\alpha_2) 
\in \R^2$, which shows that $\Psi$ is linear.

Of course, any bimeasurable application $\psi$ from $(\Omega,\F_{\infty})$ to
$(\Omega',\F_{\infty}')$ which sends $\P$ on $\P'$ induces an isomorphism
$\Psi$ from $(\Omega,\mathcal{A},\P,\F)$ into $(\Omega',\mathcal{A}',\P',\F')$,
defined by $\Psi(X) = X \circ \phi^{-1}$.
Yet, the converse is not true: an isomorphism of filtered spaces from
$(\Omega,\F_{\infty})$ to $(\Omega',\F_{\infty}')$ is not
necessarily associated to some bimeasurable application from
$\Omega$ to $\Omega'$ which sends $\P$ on $\P'$.

As a matter of fact, the most interesting objects associated to a
probability space $(\Omega,\mathcal{A},\P)$ are the random variables
and not the elements of $\Omega$.
Note that for any sequence $(X_n)_{n \le 0}$ of random variables
defined on $(\Omega,\mathcal{A},\P)$, the filtrations
which are isomorphic to the natural filtration of $(X_n)_{n \le 0}$
are exactly the filtrations of the copies of $(X_n)_{n \le 0}$
on arbitrary probability spaces.

\subsection{Immersion of filtrations}

Let $\F = (\F_n)_{n \le 0}$ and $\G = (\G_n)_{n \le 0}$ be
filtrations on $(\Omega,\mathcal{A},\P)$.

\begin{defi}
One says that $\F$ is immersed into $\G$,
if, for every $n \le 0$, $\F_n \subset \G_n$ and $\F_n$
is independent of $\G_{n-1}$ conditionally on $\F_{n-1}$.
Equivalently, $\F$ is immersed into $\G$ if and only if
every martingale in $\F$ is still a martingale in $\G$.
\end{defi}

Immersion is {\Rd } stronger than mere inclusion. If $\F$ is
immersed into $\G$, the additional information contained
in $\G$ cannot give information on $\F$ in advance:
intuitively, the independence of $\F_n$ and $\G_{n-1}$
conditionally on $\F_{n-1}$ means that $\G_{n-1}$ gives no
more information on $\F_n$ than $\F_{n-1}$ does.

The notion of immersion is implicitely present in many usual
situations. For instance, when one considers a Markov
process $X$ {\it in some filtration $\gc$},
it means that the natural filtration of $X$ is immersed in $\gc$.

\subsection{Immersibility and standardness}

The notion of immersion can be weakened to provide
a notion invariant by isomorphism.

\begin{defi}
Let $\F = (\F_n)_{n \le 0}$ and $\G' = (\G_n')_{n \le 0}$ be
filtrations on $(\Omega,\mathcal{A},\P)$ and $(\Omega',\mathcal{A}',\P')$.
One says that $\F$ is immersible into $\G'$ if there exists a filtration
$\F'$ on $(\Omega',\mathcal{A}',\P')$, isomorphic to $\F$, such that
$\F'$ is immersed into $\G'$.
\end{defi}

We can now define the standardness of filtrations.

\begin{defi}
A filtration is standard if it is immersible into a product-type filtration.
\end{defi}

Because of Kolmogorov's 0-1 law, any filtration must have a trivial tail 
$\sigma$-field in order to be standard, but this necessary condition
is not sufficient. In~\cite{Vershik}, Vershik established two
different characterisations of standardness in the context of decreasing
sequences of measurable partitions, which were extended and reformulated 
into a probabilistic language and called Vershik's ``first level'' and
``second level'' criteria by \'Emery and
Schachermayer~\cite{Emery-Schachermayer}.
\'Emery and Schachermayer
also introduced a new standardness criterion, namely the I-cosiness criterion.

\subsection{I-cosiness criterion}

Let $\F = (\F_n)_{n \le 0}$ be a filtration on $(\Omega,\mathcal{A},\P)$.

\begin{defi}
Let $R$ be any $\fc_0$-measurable real random variable $R$.
One says that $R$ satisfies I-cosiness criterion for $(\F_n)_{n \le 0}$
(to abbreviate, we say that I($R$) holds) if for any positive real
number $\delta$, there exists a probability space
$(\overline{\Omega},\overline{\mathcal{A}},\overline{\P})$ supplied
with two filtrations $\F'$ and $\F''$ such that:
\begin{itemize}
\item the filtrations $\F'$ and $\F''$ are isomorphic to the filtration $\F$;
\item the filtrations $\F'$ and $\F''$ are immersed into $\F' \vee \F''$;
\item there exists an integer $n_0<0$ such that the $\sigma$-fields
$\F'_{n_0}$ and $\F''_{n_0}$ are independent;
\item the copies $R'$ and $R''$ of $R$ given by the isomorphisms
of the first condition are such that $\overline{\P}[|R'-R''| \ge \delta]
\le \delta$.
\end{itemize}
One says that $\F$ is I-cosy when I($R$) holds for every
$R \in L^0(\Omega,\fc_0,\P)$.
\end{defi}

The definition of I-cosiness was implicitly used by Smorodinsky in
\cite{Smorodinsky} to prove that the dyadic split-words filtration is
not standard (although Smorodinsky uses a different terminology).
The I stands for independence, to distinguish I-cosiness from other
variants of cosiness.

Intuitively, the conditions defining I($R$) mean that one can couple two
copies of $\F$ in a non-anticipative way so that old enough independent
initial conditions have weak influence on the final value of $R$.

Laurent noticed that if $I(R)$ holds, then $I(\phi(R))$ holds for every Borel
function $\phi$ from $\R$ to $\R$. Hence, to prove that $\F$ is I-cosy,
it is sufficient to check that $I(R)$ for {\it one} real random variable
generating $\F_0$.

It is also sufficient and sometimes handful to check $I(R)$ for all
random variables with values in an arbitrary finite set,
with the discrete distance $\one_{[R' \ne R'']}$ replacing $|R'-R''|$
in the definition of $I(R)$.

I-cosiness provides a standardness criterion.

\begin{theo}~\label{I-cosiness criterion}
{\bf (\'Emery and Schachermayer~\cite{Emery-Schachermayer})}
$\F$ is standard if and only if $\F$ is I-cosy.
\end{theo}

\end{document}